\documentclass[12pt]{article}
\usepackage{amssymb}
\usepackage{amsmath,amsthm}
\usepackage[latin1]{inputenc}
\usepackage{graphicx}
\usepackage{hyperref}
\usepackage{enumerate}
\usepackage{multirow}
\usepackage{tikz}
\usetikzlibrary{decorations.pathreplacing}
\usepackage{color}
\usetikzlibrary{shapes.geometric}
\usepackage{mathrsfs}
\usepackage{caption}
\usepackage{verbatim}
\usepackage{algorithm, algpseudocode}
\usepackage{subfigure}

\hypersetup{colorlinks=true, linkcolor=blue, citecolor=blue, urlcolor=blue}

\newtheorem{remark}{Remark}

\newtheorem{lemma}[remark]{Lemma}
\newtheorem{theorem}[remark]{Theorem}
\newtheorem{proposition}[remark]{Proposition}
\newtheorem{corollary}[remark]{Corollary}

\title{The $k$-metric dimension of corona product graphs}
\author{A. Estrada-Moreno$^{(1)}$, I. G. Yero$^{(2)}$, J. A. Rodr\'{\i}guez-Vel\'{a}zquez$^{(1)}$\\
$^{(1)}${\small Departament d'Enginyeria Inform\`atica i Matem\`atiques,}\\
{\small Universitat Rovira i Virgili,}  {\small Av. Pa\"{\i}sos
Catalans 26, 43007 Tarragona, Spain.} \\{\small
 alejandro.estrada\@@urv.cat, juanalberto.rodriguez\@@urv.cat}\\
    $^{(2)}${\small Departamento de Matem\'aticas, Escuela Polit\'ecnica Superior de Algeciras}\\
{\small Universidad de C\'adiz,} {\small
Av. Ram\'on Puyol s/n, 11202 Algeciras, Spain.} \\ {\small
ismael.gonzalez\@@uca.es}
}

\date{}

\begin{document}
\maketitle

\begin{abstract}
Given a connected simple graph $G=(V,E)$, and a positive integer $k$, a set $S\subseteq V$ is said to be a $k$-metric generator for $G$ if and only if for any pair of different vertices $u,v\in V$, there exist at least $k$ vertices $w_1,w_2,...,w_k\in S$ such that $d_G(u,w_i)\ne d_G(v,w_i)$, for every $i\in \{1,...,k\}$, where $d_G(x,y)$ is the length of a shortest path between $x$ and $y$. A $k$-metric generator of minimum cardinality in $G$ is called a $k$-metric basis and its cardinality, the $k$-metric dimension of $G$. In this article we study the $k$-metric dimension of corona product graphs $G\odot\mathcal{H}$, where $G$ is a graph of order $n$ and $\mathcal{H}$ is a family of $n$ non-trivial graphs. Specifically, we give some necessary and sufficient conditions for the existence of a $k$-metric basis in a connected corona graph. Moreover, we obtain tight bounds and closed formulae for the $k$-metric dimension of connected corona graphs.
\end{abstract}

{\it Keywords:}  $k$-metric generator; $k$-metric dimension; $k$-metric basis; $k$-metric dimensional graphs; corona product graphs.

{\it AMS Subject Classification numbers:}   05C12; 05C76

\section{Introduction}

The concept of $k$-\textit{metric generator}  was introduced by the authors of this paper in \cite{Estrada-Moreno2013} as a generalization of the standard concept of metric generator. In graph theory, the notion of metric generator was previously given by Slater in \cite{Slater1975,Slater1988}, where the metric generators were called \emph{locating sets}, and also, independently by Harary and Melter in \cite{Harary1976}, where the metric generators were called \emph{resolving sets}. These characteristic sets were introduced in connection with the problem of uniquely determining the location of an intruder in a network. After that, several other applications of metric generators have been presented. For instance, applications to the navigation of robots in networks are discussed in \cite{Khuller1996}, and applications to chemistry are discussed in \cite{Johnson1993, Johnson1998}. Moreover, this issue has been studied in other papers including, for instance, \cite{Caceres2007, Chartrand2000, Haynes2006, Okamoto2010, Yero2011}.

For more realistic settings, $k$-metric generators allow to study a more general approach of locating problems. Consider, for instance, some robots which are navigating, moving from node to node of a network. Since on a graph there is not the concept of direction nor that of visibility, we assume that robots have communication with a set of landmarks $S$ (a subset of nodes), which provides them the distance to the landmarks in order to facilitate the navigation. In this sense, one aim is that each robot is uniquely determined by the landmarks. Suppose that in a specific moment there are two robots $x,y$, whose positions are only distinguished by one landmark $s\in S$. If the communication between $x$ and $s$ is ``unexpectedly blocked'', then the robot $x$ will get ``lost'' in the sense that it can assume that it has the position of $y$. So, for security reasons, we will consider a set of landmarks, where each pair of nodes is distinguished by at least $k\ge 2$ landmarks, {\em i.e.}, to take $S$ as a $k$-metric generator for $k\ge 2$.

Given a simple and connected graph $G=(V,E)$ we denote by $d_G(x ,y)$ the distance between $x,y\in V$. A  set $S\subset V$ is said to be a \emph{metric generator} for $G$ if for any pair of vertices  $x,y\in V$ there exists  $s\in S$ such that $d_G(s,x)\ne d_G(s,y)$ (in this case we say that the pair $x,y$ is  \textit{distinguished} by $s$).
 A \textit{minimum metric generator} is a metric generator with the smallest possible cardinality among all the metric generators for $G$.
A minimum metric generator is called a \emph{metric basis}, and its cardinality, the \emph{metric dimension} of $G$, denoted by $\dim(G)$. Given $S=\{s_{1}, s_{2}, \ldots, s_{d}\}\subseteq V(G)$, we refer to the $d$-vector (ordered $d$-tuple) $r(u|S)=\left(d_{G}(u,s_{1}),d_{G}(u,s_{2}),\ldots,d_{G}(u,s_{d})\right)$ as the \emph{metric representation} of $u$ with respect to $S$. In this sense, $S$ is a metric generator for $G$ if and only if for every pair of different vertices $u,v$ of $G$, it follows $r(u|S)\neq r(v|S)$.

Now, in a more general setting, given a positive integer $k$, a set $S\subseteq V$ is said to be a \emph{$k$-metric generator} for $G$ if and only if any pair of vertices of $G$ is distinguished by at least $k$ elements of $S$, {\em i.e.}, for any pair of different vertices $u,v\in V$, there exist at least $k$ vertices $w_1,w_2,...,w_k\in S$ such that
\begin{equation}\label{conditionDistinguish}
d_G(u,w_i)\ne d_G(v,w_i),\; \mbox{\rm for every}\; i\in \{1,...,k\}.
\end{equation}

Obviously, $1$-metric generators are the standard metric generators (resolving sets or locating sets as defined in \cite{Harary1976} or \cite{Slater1975}, respectively). By analogy to the standard case, a $k$-metric generator of minimum cardinality will be  called a \emph{$k$-metric basis} of $G$ and its cardinality, the \emph{$k$-metric dimension} of $G$, which will be denoted by $\dim_{k}(G)$. Notice that every $k$-metric generator $S$ satisfies that $|S|\geq k$ and, if $k>1$, then $S$ is also a $(k-1)$-metric generator.

In practice,  the problem of checking if a set $S$ is a $1$-metric generator is reduced to check condition (\ref{conditionDistinguish}) only for those vertices $u,v\in V- S$, as every vertex in $S$ is distinguished at least by itself. Also, if $k=2$, then condition (\ref{conditionDistinguish}) must be checked only for those pairs having at most one vertex in $S$, since two vertices of $S$ are distinguished at least by themselves. Nevertheless, if $k\ge 3$, then condition (\ref{conditionDistinguish}) must be checked for every pair of different vertices of the graph.

It was shown in \cite{Yero2013c}, that the problem of computing the $k$-metric dimension of a graph is NP-complete (the case $k=1$ was previously studied in \cite{Khuller1996}). It is therefore motivating to find the $k$-metric dimension for special classes of graphs or good bounds on this invariant. Specifically, for the case of product graphs, it would be desirable to reduce the problem of computing the $k$-metric dimension of a product graph into computing the $k$-metric dimension of the factor graphs.

Studies about the metric dimension of product graphs were initiated in \cite{Caceres2007,Peters-Fransen2006}, where several tight bounds and closed formulae for the metric dimension of Cartesian product graphs were presented. After that, the metric dimension of corona graphs, rooted product graphs, lexicographic product graphs and strong product graphs was studied in \cite{Yero2011}, \cite{Yero2013b}, \cite{JanOmo2012,Saputro2013} and \cite{Rodriguez-Velazquez-et-al2014}, respectively. In this work we continue with the study of the $k$-metric dimension of the corona product graphs. To this end, we introduce some notation and terminology.

If two vertices $u,v$ are adjacent in $G=(V,E)$, then we write $u\sim v$ or  $uv\in E(G)$. Given  $x\in V(G)$, we define $N_{G}(x)$ as the \emph{open neighborhood} of $x$ in $G$, \textit{i.e.},  $N_{G}(x)=\{y\in V(G):x\sim y\}$. The \emph{closed neighborhood}, denoted by $N_{G}[x]$, equals $N_{G}(x)\cup \{x\}$. If there is no  ambiguity, we will simply write  $N(x)$ or $N[x]$. We also refer to the degree of $v$ as $\delta(v)=|N(v)|$.
For a non-empty set $S \subseteq V(G)$, and a vertex $v \in V(G)$, $N_S(v)$ denotes the set of neighbors that $v$ has in $S$, {\it i.e.}, $N_S(v) = S\cap N(v)$. As usual, we denote by $A\triangledown B=(A\cup B)- (A\cap B)$ the symmetric difference of two sets  $A$ and $B$.

We now recall that the \emph{join graph} $G+H$ of the graphs $G=(V_{1},E_{1})$ and $H=(V_{2},E_{2})$ is the graph with vertex set $V(G+H)=V_{1}\cup V_{2}$ and edge set $E(G+H)=E_{1}\cup E_{2}\cup \{uv\,:\,u\in V_{1},v\in V_{2}\}$.

Let $G$ be a graph of order $n$ and let ${\cal H}=\{H_1,H_2,...,H_n\}$ be a family of graphs. The corona product graph $G\odot {\cal H}$ is defined as the graph obtained from $G$ and ${\cal H}$ by taking one copy of $G$ and joining by an edge each vertex of $H_i$ with the $i^{th}$-vertex of $G$, \cite{Frucht1970}. Notice that the particular case of corona graph $K_1\odot H$ is isomorphic to the join graph $K_1+H$. From now on  we will denote by $V=\{v_1,v_2,\ldots, v_{n}\}$ the set of vertices of $G$ and by $H_i=(V_i, E_i)$ the graphs belonging to ${\cal H}$.  So the vertex set of $G\odot {\cal H}$ is $V(G\odot {\cal H})=V\cup\left(\bigcup_{i=1}^n V_i\right)$. Also, the order of the graph $H_i\in {\cal H}$ will be denoted $n_i$. In particular, if
every $H_i\in {\cal H}$ holds that $H_i\cong H$, then we will use the notation $G\odot H$ instead of  $G\odot {\cal H}$. In this work, the remaining definitions will be given the first time that the concept appears in the text.

Several results about the $k$-metric dimension of  corona product graphs, $G\odot {\cal H}$, where at least one graph belonging to ${\cal H}$ is trivial, are presented in \cite{Estrada-Moreno2013a}. Thus, the aim of this paper is to study the case where all graphs belonging to ${\cal H}$ are non-trivial.

The paper is organized as follows: in Section \ref{sectionDimensionalCorona}  we give some necessary and sufficient conditions for the existence of a $k$-metric basis for an arbitrary connected  corona graph $G\odot\mathcal{H}$. So, we determine the range of possible values for $k$, where $\dim_k(G\odot\mathcal{H})$ makes sense. In Section  \ref{KmetricDimCoronaGraphs} we obtain tight bounds and closed formulae for the $k$-metric dimension of  corona graphs where the values of  $k$ cover the range stated in Section \ref{sectionDimensionalCorona}.

\section{$k$-metric dimensional corona graphs}\label{sectionDimensionalCorona}

A connected graph $G$ is said to be a \emph{$k'$-metric dimensional graph} if $k'$ is the largest integer such that there exists a $k'$-metric basis \cite{Estrada-Moreno2013}. Notice that if $G$ is a $k'$-metric dimensional graph, then for each positive integer $k\le k'$, there exists at least one $k$-metric basis for $G$, \textit{i.e.,} $\dim_{k}(G)$ makes sense for $k\in \{1,...,k'\}$. Since for every pair of vertices $x,y$ of a graph $G$, we have that they are distinguished at least by themselves, it follows that the whole vertex set $V(G)$ is a $2$-metric generator for $G$ and, as a consequence, it follows that every graph $G$ is $k'$-metric dimensional for some $k'\ge 2$. On the other hand, for any connected graph $G$ of order $n>2$, there exists at least one vertex $v\in V(G)$ such that $\delta(v)\ge 2$. Since $v$ does not distinguish any pair $x,y\in N_G(v)$,  there is no $n$-metric dimensional graph of order $n>2$.

We first present a characterization of $k$-metric dimensional graphs obtained in \cite{Estrada-Moreno2013}. To do so, we need  some additional terminology. Given two vertices $x,y\in V(G)$, we say that the set of \textit{distinctive vertices} of $x,y$ is $${\cal D}_G(x,y)=\{z\in V(G): d_{G}(x,z)\ne d_{G}(y,z)\}$$
and, the set of \emph{non-trivial distinctive vertices} of  $x,y$ is $${\cal D}^*_G(x,y)={\cal D}_G(x,y)-\{x,y\}.$$

\begin{theorem} \label{theokmetric}  {\rm \cite{Estrada-Moreno2013}}
A connected graph  $G$ is $k$-metric dimensional  if and only if $k=\displaystyle\min_{x,y\in V(G)}\{\vert {\cal D}_G(x,y)\vert\} .$
\end{theorem}

Two vertices $x,y$ are called \emph{false twins} if $N(x)=N(y)$, and $x,y$ are called \emph{true twins} if $N[x]=N[y]$. Two vertices $x,y$ are \emph{twins} if they are false twin vertices or true twin vertices. Notice that two vertices $x,y$ are twins if and only if
${\cal D}^*_G(x,y)=\emptyset$, {\em i.e.}, ${\cal D}_G(x,y)=\{x,y\}$. We also say that a vertex $x$ is a twin, if there exists other vertex $y$ such that $x,y$ are twins.

\begin{corollary}\label{coro2dimensional} {\rm \cite{Estrada-Moreno2013}}
A connected graph $G$ of order $n\geq 2$ is $2$-metric dimensional if and only if $G$ has  twin vertices.
\end{corollary}

If there exists a graph $H_i\in\mathcal{H}$ such that $H_i$ has twin vertices, then it follows that for any graph $G$,  the corona graph $G\odot\mathcal{H}$ has twin vertices. Also notice that any two vertices of $G$ are not twins in $G\odot\mathcal{H}$. Therefore, according to Corollary \ref{coro2dimensional} we deduce the following result.

\begin{remark}\label{remarkCorona2}
For any connected graph $G$ of order $n$ and any family $\mathcal{H}$ composed by $n$ connected non-trivial graphs, the corona graph $G\odot\mathcal{H}$ is $2$-metric dimensional if and only if there exists a $2$-metric dimensional graph $H_i\in\mathcal{H}$.
\end{remark}

\begin{corollary}
Let $G$ be a connected   graph. Then,
\begin{enumerate}[{\rm (i)}]
\item For $n\ge 2$, the graph $G\odot K_n$ is $2$-metric dimensional.

\item The graphs   $G\odot P_3$ and  $G\odot C_4$  are  $2$-metric dimensional.
\end{enumerate}
\end{corollary}

\subsection{$k$-metric dimensional graphs of the form  $G\odot \mathcal{H}$, where $G\not\cong K_1$.}

Given a connected non-trivial graph $H$, we define $$\mathcal{C}(H)=\min_{x,y\in V(H)}\{\vert N_H(x)\triangledown N_H(y) \cup\{x,y\}\vert\}.$$ According to that notation, for a family of connected non-trivial graphs $\mathcal{H}$, we define $$\mathcal{C}(\mathcal{H})=\min_{H_i\in\mathcal{H}}\{\mathcal{C}(H_i)\}.$$

\begin{theorem}\label{theoKmetricCorona}
Let $G$ be a connected non-trivial graph of order $n$ and let $\mathcal{H}$ be a family of $n$ non-trivial graphs. Then, $G\odot \mathcal{H}$ is $k$-metric dimensional if and only if $k=\mathcal{C}(\mathcal{H})$.
\end{theorem}

\begin{proof}
We claim that $\mathcal{C}(\mathcal{H})=\displaystyle\min_{x,y\in V(G\odot \mathcal{H})}\{\vert {\cal D}_{G\odot \mathcal{H}}(x,y)\vert\}$.
Notice that, for every $u,v\in V(H_i)$, we have that $|N_{H_i}(u)\triangledown N_{H_i}(v)|\le |V(H_i)|$. Let $x,y$ be two different vertices of $G\odot \mathcal{H}$. We consider the following cases.
\\
\\
\noindent Case 1. If $x\in V_i$ and $y\in V_j$, $i\ne j$, then $\mathcal{D}_{G\odot \mathcal{H}}(x,y)=\displaystyle\bigcup_{v_l\in \mathcal{D}_G(v_i,v_j)}(V_l\cup \{v_l\})$.
\\
\\
\noindent Case 2. If $x,y\in V$, then we assume that $x=v_i$ and $y=v_j$. So, it follows that $\mathcal{D}_{G\odot \mathcal{H}}(x,y)=\displaystyle\bigcup_{v_l\in \mathcal{D}_G(v_i,v_j)}(V_l\cup \{v_l\})$.
\\
\\
\noindent Case 3. If $x\in V_i$ and $y\in V$, then $y=v_j$ for some $j\in \{1,...,n\}$ and we consider the following. If $j=i$, then $\mathcal{D}_{G\odot \mathcal{H}}(x,y)=V(G\odot \mathcal{H})-N_{H_i}(x)$. Now, if $j\ne i$, then we have $\mathcal{D}_{G\odot \mathcal{H}}(x,y) \supseteq V_j$.
\\
\\
\noindent Case 4. If $x,y\in V_i$, then $\mathcal{D}_{G\odot \mathcal{H}}(x,y)=(N_{H_i}(x)\triangledown N_{H_i}(y))\cup\{x,y\}.$

Now, notice that from Cases $1$, $2$ and $3$, $|\mathcal{D}_{G\odot \mathcal{H}}(x,y)|\ge \displaystyle\min_{H_i\in\mathcal{H}}\lbrace|V_i|\rbrace\ge \min_{H_i\in\mathcal{H}}\{\mathcal{C}(H_i)\}=\mathcal{C}(\mathcal{H})$. Also, in Case 4, for every $x,y\in V_i$ we have that $\vert\mathcal{D}_{G\odot \mathcal{H}}(x,y) \vert=\vert (N_{H_i}(x)\triangledown N_{H_i}(y))\cup\{x,y\}\vert\ge \displaystyle\min_{H_j\in\mathcal{H}}\{\mathcal{C}(H_j)\}=\mathcal{C}(\mathcal{H})$.
Thus, $$\mathcal{C}(\mathcal{H})\le \min_{x,y\in V(G\odot \mathcal{H})}\{\vert {\cal D}_{G\odot \mathcal{H}}(x,y)\vert\}.$$
On the other hand, we consider the following.
\begin{align*}
\min_{x,y\in V(G\odot \mathcal{H})}\{\vert {\cal D}_{G\odot \mathcal{H}}(x,y)\vert\}&
\le \min_{x,y\in V(G\odot \mathcal{H})-V(G)}\{\vert {\cal D}_{G\odot \mathcal{H}}(x,y)\vert\}\\
&\le \min_{H_i\in\mathcal{H}}\lbrace\min_{x,y\in V_i}\{\vert {\cal D}_{G\odot \mathcal{H}}(x,y)\vert\}\rbrace\\
&=\min_{H_i\in\mathcal{H}}\lbrace\min_{x,y\in V_i}\{\vert N_{H_i}(x)\triangledown N_{H_i}(y)\cup\{x,y\}\vert\}\rbrace\\
&=\min_{H_i\in\mathcal{H}}\{\mathcal{C}(H_i)\}\\
&=\mathcal{C}(\mathcal{H}).
\end{align*}
Therefore $\mathcal{C}(\mathcal{H})=\displaystyle\min_{x,y\in V(G\odot \mathcal{H})}\{\vert {\cal D}_{G\odot \mathcal{H}}(x,y)\vert\}$ and, by Theorem \ref{theokmetric}, we conclude the proof.
\end{proof}

Notice that if every $H_i\in\mathcal{H}$ satisfies that $H_i\cong H$, then $\mathcal{C}(\mathcal{H})=\mathcal{C}(H)$. Thus, the following result follows from Theorem \ref{theoKmetricCorona}.

\begin{corollary}\label{coroKmetricCorona}
Let $G$ and $H$ be two connected non-trivial graphs. Then $G\odot H$ is $k$-metric dimensional if and only if $k=\mathcal{C}(H)$.
\end{corollary}

According to Theorem \ref{theoKmetricCorona}, if the corona graph $G\odot \mathcal{H}$ is $k$-metric dimensional, then the value of $k$ is independent from the connected non-trivial graph $G$. Moreover,  for any $x,y\in V_i$ it holds ${\cal D}_{H_i}(x,y)\supseteq (N_{H_i}(x)\triangledown N_{H_i}(y))\cup \{x,y\}$. Therefore, by Theorems \ref{theokmetric} and  \ref{theoKmetricCorona} we deduce the following result.

\begin{proposition}\label{IHkDimThenCoronakDim}
Let $G\odot \mathcal{H}$ be a $k$-metric dimensional graph such that
$G$ is a connected non-trivial graph and  $\mathcal{H}=\{H_1,H_2,...,H_n\}$ is a family of connected non-trivial graphs, where  $H_i$ is $k_i$-metric dimensional for $i\in \{1,...,n\}$. Then the following assertions hold:

\begin{enumerate}[{\rm (i)}]
\item $k\le \displaystyle\min_{i\in \{1,...,n\}}\{k_i\}.$

\item   $k=k_j$ if and only if $ \displaystyle\min_{i\in \{1,...,n\}}\left\{ \mathcal{C}(H_i) \right\}=\displaystyle\min_{x,y\in V_j}\{\vert {\cal D}_{H_j}(x,y)\vert\}.$

\item If $k=k_j$, then $\mathcal{C}(H_j)=\displaystyle\min_{x,y\in V_j}\{\vert {\cal D}_{H_j}(x,y)\vert\}$.
\end{enumerate}
\end{proposition}

If a graph $H$ has diameter  $D(H)\le 2$, then for every $x,y\in V(H)$ it holds  ${\cal D}_H(x,y)=N_H(x)\triangledown N_H(y)\cup \{x,y\}$. Thus, the following result is deduced.

\begin{corollary}
Let $G\odot \mathcal{H}$ be a $k$-metric dimensional graph where
$G$ is a connected non-trivial graph and  $\mathcal{H}=\{H_1,H_2,...,H_n\}$ is a family of  graphs such that  $H_i$ is $k_i$-metric dimensional and $D(H_i)\le 2$, for every $i\in \{1,...,n\}$. Then
$k= \displaystyle\min_{i\in \{1,...,n\}}\{k_i\}.$
\end{corollary}

The \textit{girth} $g(H)$ of a graph $H$ is the length of a shortest cycle contained in $H$. Now, if  $g(H)\ge 5$, then for every $x,y\in V(H)$  we have that either $\vert N_H(x)\cap N_H(y) \vert=1$ or $\vert N_H(x)\cap N_H(y) \vert=0$. Hence, it follows that the next result as a consequence of Theorem \ref{theoKmetricCorona}.

\begin{corollary} \label{cuello-ge5}
Let
$G$ be a connected non-trivial graph of order $n$ and  let $\mathcal{H}=\{H_1,H_2,...,H_n\}$ be a family of  $\delta$-regular graphs where   $g(H_i)\ge 5$, for every $i\in \{1,...,n\}$. Then
$G\odot \mathcal{H}$ is a $2\delta$-metric dimensional graph.
\end{corollary}

We would point out the following particular case of Corollary \ref{cuello-ge5}.

\begin{remark}
Let $G$ be a connected non-trivial graph. Then, for $n\ge 5$, the graph $G\odot C_n$ is $4$-metric dimensional.
\end{remark}

An \textit{end-vertex} of a graph $H$ is a vertex of degree one and a \textit{support vertex} is a vertex that is adjacent to an end-vertex. If $x\in V(H)$ is an end-vertex and $y\in V(H)$ is a support vertex of degree two which is adjacent to $x$, then $\vert N_H(x)\triangledown N_H(y)\cup \{x,y\}\vert =3$. Thus, from Corollary \ref{coro2dimensional} and  Theorem \ref{theoKmetricCorona} we deduce the following result.

\begin{proposition}
Let $G$ be  a connected non-trivial graph of order $n$ and  let $\mathcal{H}$ be a family of $n$ connected non-trivial graphs such that no graph belonging to $\mathcal{H}$ has twin vertices. If there exists $H\in \mathcal{H}$, having an end-vertex whose support vertex has degree two, then $G\odot \mathcal{H}$ is  a $3$-metric dimensional graph.
\end{proposition}

An interesting particular case of the result above is when the family $ \mathcal{H}$ contains a path $P_r$ of order $r\ge 4$ and  no graph belonging to $\mathcal{H}$ has twin vertices.
In such a case $G\odot \mathcal{H}$ is  a $3$-metric dimensional graph.

\subsection{$k$-metric dimensional graphs of the form  $K_1+ H$}

\begin{proposition}\label{kMectricDimK1H}
Let $H$ be a graph of order $n'\ge 2$ and maximum degree $\Delta(H)$. The graph $K_1+ H$ is $k$-metric dimensional if and only if $k=\min\{\mathcal{C}(H),n'-\Delta(H)+1\}$.
\end{proposition}

\begin{proof}
Let $v$ be the vertex of $K_1$. Now, let $x,y$ be two different vertices of $K_1+ H$. If $x,y\in V(H)$, then $\mathcal{D}_{K_1+H}(x,y)=N_H(x)\triangledown N_H(y)\cup \{x,y\}$. If $x=v$ and $y\in V(H)$, then $\mathcal{D}_{K_1+H}(x,y)=(V(H)-N_{H}(y))\cup \{x\}$. Therefore, by Theorem \ref{theokmetric}, the result follows.
\end{proof}

We next point out some consequences of Proposition \ref{kMectricDimK1H}.

\begin{corollary}\label{H,K_1+H}
Let $H$ be a non-trivial graph. If $H$ is $k$-metric dimensional and    $K_1+H$ is $k'$-metric dimensional, then $k'\le k$.
\end{corollary}

\begin{proof}
By Proposition \ref{kMectricDimK1H} we have that if $K_1+H$ is a $k'$-metric dimensional graph, then  $k'\le \mathcal{C}(H)$. Since, for any $x,y\in V(H)$ we have ${\cal D}_{H}(x,y)\supseteq N_{H}(x)\triangledown N_{H}(y)\cup \{x,y\}$, we deduce that  if $H$ is $k$-metric dimensional, then $\mathcal{C}(H)\le k$ and, as a consequence, $k'\le k$.
\end{proof}

\begin{corollary}
For any connected graph $H$ of order $n'\ge 2$ and maximum degree $n'-1$, the graph $K_1+H$ is $2$-metric dimensional.
\end{corollary}

Notice that the corollary above may be also derived from Corollary \ref{coro2dimensional}.

\begin{corollary}
Let $H$ be a connected graph of order $n'\ge 4$ and maximum degree $n'-2$. If $H$ does not contain twin vertices, then $K_1+H$ is $3$-metric dimensional.
\end{corollary}

\begin{proof}
Since $H$ does not contain twin vertices, for every $x,y\in V(H)$ there exists $z\in V(H)-\{x,y\}$ such that $z\in N(x)\triangledown N(y)$. Thus, $\mathcal{C}(H)\ge 3$. Now, since $n'-\Delta(H)+1=3$, by Proposition \ref{kMectricDimK1H} we can deduce the result.
\end{proof}

The {\em wheel graph} $W_{1,n}$ is the join  graph $K_1+C_n$ and the {\em fan graph} $F_{1,n}$ is the join graph $K_1+P_n$.

\begin{corollary}\label{dimensional-fan-wheel}
For any $n\ge 4$, the fan graph  $F_{1,n}$ is $3$-metric dimensional, and for any $n\ge 5$, the wheel graph  $W_{1,n}$ is $4$-metric dimensional.
\end{corollary}

By Corollary \ref{coroKmetricCorona} and Proposition \ref{kMectricDimK1H} we deduce the following remark.

\begin{remark}\label{Comparakenproductoyk1+H}
Let $G$ be a connected graph of order $n\ge 2$ and let $\mathcal{H}$ be a family of $n$ non-trivial connected  graphs.
If for every $H_i\in \mathcal{H}$ the graph $K_1+H_i$ is $k_i$-metric dimensional and $G\odot\mathcal{H}$ is $k$-metric dimensional, then
 $k\ge \displaystyle\min_{i\in \{1,...,n\}}\{k_i\}$.
\end{remark}

We conclude this section with  a property
on the   $(n'-\Delta(H)+1)$-metric bases of  $K_1+H$.

\begin{proposition}\label{propK1Belong}
Let $H$ be a non-trivial graph of order $n'$. If $K_1+H$ is $(n'-\Delta(H)+1)$-metric dimensional, then the vertex of $K_1$ belongs to every $(n'-\Delta(H)+1)$-metric basis of $K_1+H$.
\end{proposition}

\begin{proof}
Let $v$ be the vertex of $K_1$.  Notice that for every $x\in V(H)$, we have $$\mathcal{D}_{K_1+H}(x,v)=\left(V(H)-N_H(x)\right)\cup \{v\}.$$
For every $x\in V(H)$ such that $N_H(x)=\Delta(H)$ we have that $n'-\Delta(H)+1=|\left(V(H)-N(x)\right)\cup \{v\}|=|\mathcal{D}_{K_1+H}(x,v)|$. Thus, for any $(n'-\Delta(H)+1)$-metric basis $B$ we have  $\mathcal{D}_{K_1+H}(x,v)\subseteq B$ and, since $v\in \mathcal{D}_{K_1+H}(x,v)$, we conclude that $v\in B$.
\end{proof}

By Proposition \ref{propK1Belong} we deduce that if the vertex of $K_1$ does not belong to any $k$-metric basis of $K_1+H$, then $K_1+H$ is not $(n'-\Delta(H)+1)$-metric dimensional. Thus, by Proposition \ref{kMectricDimK1H} we obtain the following result.

\begin{lemma}\label{K1_H_is_CH_dimensional}
Let $H$ be a non-trivial graph. If the vertex of $K_1$ does not belong to any $k$-metric basis of $K_1+H$, then $K_1+H$ is $\mathcal{C}(H)$-metric dimensional.
\end{lemma}

\section{The $k$-metric dimension of corona product graphs}
\label{KmetricDimCoronaGraphs}

Once we have presented several results on $k$-metric dimensional corona graphs, in this section we compute or bound the $k$-metric dimension of  corona graphs. To do so, we need to introduce the necessary terminology and some useful tools like the following straightforward lemma.

\begin{lemma}\label{lemmaBelongKBasis}
Let $G$ be a connected graph and let $x,y\in V(G)$. If  $B$ is a $k$-metric basis  of $G$ and $\vert {\cal D}_G(x,y)\vert =k$, then  ${\cal D}_G(x,y)\subseteq B.$
\end{lemma}

Given a $k$-metric dimensional graph $G$, we define  ${\cal D}_k(G)$ as the set obtained by the union of the sets of distinctive vertices ${\cal D}_G(x,y)$ whenever $\vert{\cal D}_G(x,y)\vert=k$, {\it i.e.}, $${\cal D}_k(G)=\bigcup_{\vert {\cal D}_G(x,y)\vert=k}{\cal D}_G(x,y).$$

\begin{corollary}\label{corollaryBelongKBasis}
Let $G$ be a $k$-metric dimensional graph. For any $k$-metric basis $B$ of a graph $G$ it holds ${\cal D}_k(G)\subseteq B .$
\end{corollary}

\begin{theorem}\label{theoDimkn} {\rm \cite{Estrada-Moreno2013}}
Let $G$ be a $k$-metric dimensional graph of order $n$. Then $\dim_k(G)=n$ if and only if $V(G)={\cal D}_k(G)$.
\end{theorem}

\begin{corollary}\label{remarkDim2n}{\rm \cite{Estrada-Moreno2013}}
Let $G$ be a connected graph of order $n\geq 2$. Then $\dim_{2}(G)=n$ if and only if every vertex is a twin.
\end{corollary}

\begin{lemma}\label{lemmaKCorona}
Let $G=(V,E)$ be a connected graph of order $n\ge 2$ and let $\mathcal{H}=\{H_1,H_2,...,H_n\}$ be a family of connected non-trivial graphs.  If $G\odot \mathcal{H}$ is $k'$-metric dimensional, then the following assertions hold for any $k\in \{1,...,k'\}.$

\begin{enumerate}[{\rm (i)}]
\item If $u,v\in V_i$, then $d_{G\odot \mathcal{H}}(u,x)=d_{G\odot \mathcal{H}}(v,x)$ for every vertex $x$ of $G\odot \mathcal{H}$ not belonging to $V_i$.
\item If $S$ is a $k$-metric generator for $G\odot \mathcal{H}$, then $\vert V_i\cap S\vert \ge k$ for every $i\in \{1,\ldots, n\}$.
\item If $S$ is a $k$-metric basis of $G\odot \mathcal{H}$, then $V\cap S=\emptyset$.
\item If $S$ is a $k$-metric generator for $G\odot \mathcal{H}$, then for every $i\in \{1,\ldots, n\}$, the set $S\cap V_i$ is a $k$-metric generator for $H_i$.
\end{enumerate}
\end{lemma}

\begin{proof}
\begin{enumerate}[{\rm (i)}]
\item It is straightforward.
\item Let $S$ be a $k$-metric generator for $G\odot \mathcal{H}$. Then for any pair of vertices $x,y\in V_i$ there exist at least $k$ vertices  $u\in S$ such that $d_{G\odot\mathcal{H}}(x,u) \ne d_{G\odot\mathcal{H}}(y,u)$. Thus, by (i) it follows that $\vert S\cap V_i\vert \ge k$.

\item Let $S$ be a $k$-metric basis of $G\odot \mathcal{H}$. We will show that $S'=S-V$ is a $k$-metric generator for $G\odot \mathcal{H}$. Now, let $x,y$ be two different vertices of $G\odot \mathcal{H}$. We have the following cases.
\\
\\
\noindent {Case} 1: $x,y\in V_i$. Since $S$ is a $k$-metric basis,  by (i) we conclude that $|\mathcal{D}_{G\odot \mathcal{H}}(x,y)\cap S'|\ge k$.
\\
\\
\noindent {Case} 2: $x\in V_i$ and $y\in V_j$, $i\ne j$. Notice that for every $v\in V_i\cap S'$, we have that $d_{G\odot \mathcal{H}}(x,v)\le 2 < 3\le d_{G\odot \mathcal{H}}(y,v)$. Since $|V_i\cap S'|\ge k$, we conclude that $|\mathcal{D}_{G\odot \mathcal{H}}(x,y)\cap S'|\ge k$.
\\
\\
\noindent {Case} 3: $x,y\in V$. Let $x=v_i$. Notice that for every $v\in V_i\cap S'$ we have that $d_{G\odot \mathcal{H}}(x,v)=1<1+d_{G\odot \mathcal{H}}(y,x) = d_{G\odot \mathcal{H}}(y,v)$. Since $|V_i\cap S'|\ge k$, we conclude that $|\mathcal{D}_{G\odot \mathcal{H}}(x,y)\cap S'|\ge k$.
\\
\\
\noindent {Case} 4: $x\in V_i$ and $y\in V$. If $x\sim y$, then $y=v_i$. Let $v_j\in V$, $j\ne i$. Notice that for every $v\in V_j\cap S'$ we have that $d_{G\odot \mathcal{H}}(x,v) = 1 + d_{G\odot \mathcal{H}}(y,v)>d_{G\odot \mathcal{H}}(y,v)$. Now, if $x\not\sim y = v_l$, then for every $v\in V_l\cap S'$ it follows $d_{G\odot \mathcal{H}}(x,v) = d_{G\odot \mathcal{H}}(x,y) + d_{G\odot \mathcal{H}}(y,v)>d_{G\odot \mathcal{H}}(y, v)$. Since $|V_j\cap S'|\ge k$ and $|V_l\cap S'|\ge k$, any of the choices above implies that $|\mathcal{D}_{G\odot \mathcal{H}}(x,y)\cap S'|\ge k$.

Therefore, $S'$ is a $k$-metric generator for $G\odot \mathcal{H}$. Since $S$ is a $k$-metric basis of $G\odot \mathcal{H}$, we obtain that $V\cap S=\emptyset$.
\item Let $S$ be a $k$-metric generator for $G\odot\mathcal{H}$, and let $S_i=S\cap V_i$. By (i) we deduce that for any pair of vertices $x,y\in V_i$ it holds that $|\mathcal{D}_{G\odot \mathcal{H}}(x,y)\cap S_i|\ge k$. Since $\mathcal{D}_{G\odot\mathcal{H}}(x,y)\cap S_i\subseteq\mathcal{D}_{H_i}(x,y)$, we conclude that  $S_i$ is a $k$-metric generator for $H_i$.
\end{enumerate}
\end{proof}

\subsection{The $k$ metric dimension of $G\odot \mathcal{H}$, where $G$ and the graphs belonging to $\mathcal{H}$ are non-trivial.}
Our first result is obtained as a consequence of Lemma \ref{lemmaKCorona} (iii) and (iv).

\begin{theorem}\label{theoBoundCoronaDimH}
Let $G$ be a connected graph of order $n\ge 2$ and let $\mathcal{H}$ be a family of connected non-trivial graphs. If $G\odot \mathcal{H}$ is $k'$-metric dimensional, then for every $k\in \{1,...,k'\}$,
$$\sum_{i=1}^{n}\dim_k(H_i)\le \dim_k(G\odot \mathcal{H})\le \sum_{i=1}^n|V_i|.$$
\end{theorem}

Our next result is a direct consequence of combining the lower and upper bounds of Theorem \ref{theoBoundCoronaDimH}.

\begin{corollary}\label{PropoTheUpperboundTight}
Let $G$ be a connected graph of order $n\ge 2$ and let $\mathcal{H}$ be a family of connected non-trivial graphs. If $G\odot \mathcal{H}$ is $k$-metric dimensional and $\dim_k(H_i)=|V_i|$ for every graph $H_i\in\mathcal{H}$, then $$\dim_k(G\odot \mathcal{H})=\sum_{i=1}^n |V_i|.$$
\end{corollary}

$P_4$ and $C_6$ are two examples for the graph $H$ satisfying the conditions of Corollary \ref{PropoTheUpperboundTight}. Notice that $G\odot P_4$ is $3$-metric dimensional and $\dim_3(P_4)=4$. Also, $G\odot C_6$ is $4$-metric dimensional and $\dim_4(C_6)=6$. Therefore, the next result is a particular case of Corollary \ref{PropoTheUpperboundTight}.

\begin{remark}
For any non-trivial graph $G$ of order $n$, $\dim_3(G\odot P_4)=4n$ and $\dim_4(G\odot C_6)=6n$.
\end{remark}

\begin{theorem}\label{coroTwinBound2}
Let $G$ be a connected graph of order $n\ge 2$, and let $\mathcal{H}$ be a family of connected non-trivial graphs. Then, every $H_i\in\mathcal{H}$ is composed by twin vertices if and only if $$\dim_2(G\odot \mathcal{H})=\sum_{i=1}^n |V_i|.$$
\end{theorem}

\begin{proof}
Suppose that every $H_i\in\mathcal{H}$ is formed by twin vertices. By Corollary \ref{remarkDim2n}, we deduce that every $H_i\in\mathcal{H}$ holds that $\dim_2(H_i)=|V_i|$. So, by Corollary \ref{PropoTheUpperboundTight} we conclude that $\dim_2(G\odot \mathcal{H})=\sum_{i=1}^n |V_i|$.

Conversely, assume that $\dim_2(G\odot \mathcal{H})=\sum_{i=1}^n |V_i|$.
We proceed by contradiction.   Suppose that there exists $x\in V_i$ such that for every $y\in W=V_i-\{x\}$ it holds $N_{H_i}(x)\ne N_{H_i}(y)$.
In such a case, $\vert V_i\vert \ge 3$ and since $H_i$ is connected, for every $y\in W$ we have the following.
\begin{itemize}
\item If $y\sim x$, then $\vert N_{H_i}(x)\bigtriangledown N_{H_i}(y) - \{x\}\vert \ge 2$ and, as a consequence, $\vert\mathcal{D}_{G\odot \mathcal{H}}(x,y)\cap W \vert \ge 2$.
\item If $y\not\sim x$, then $\vert N_{H_i}(x)\bigtriangledown N_{H_i}(y)\vert \ge 1$ and also $y$ distinguishes the pair $x,y$. Thus, again, $\vert\mathcal{D}_{G\odot \mathcal{H}}(x,y)\cap W \vert \ge 2$.
\end{itemize}
Now, we take $S$ as a $2$-metric basis of $G\odot \mathcal{H}$. By Lemma \ref{lemmaKCorona} (iii) we have that $S\cap V=\emptyset$ and, consequently, for any $j\in \{1,...,n\}$ we have $S\cap V_j=V_j$. Also,
by Lemma \ref{lemmaKCorona} (i), every pair of vertices of $H_j$ is only distinguished by vertices of $H_j$. Therefore, $S'=W\cup \left(\bigcup_{j\ne i}V_j \right)$
is a $2$-metric generator for $G\odot \mathcal{H}$ and $\vert S' \vert < \sum_{i=1}^n |V_i|= dim_2(G\odot \mathcal{H}) $, which is a contradiction.
\end{proof}

Next we present another case where the lower bound of Theorem \ref{theoBoundCoronaDimH} is achieved.

\begin{theorem}\label{theoEqualCoronaDimH}
Let $G$ be a connected graph of order $n\ge 2$ and let $\mathcal{H}$ be a family of $n$ non-trivial graphs such that every $H_i\in \mathcal{H}$ is $k_i$-metric dimensional and $D(H_i)\le 2$. If $k'=\displaystyle\min_{i\in 1,...,n} \{k_i\}$, then for every $k\in\{1,..., k'\}$,
$$\dim_k(G\odot \mathcal{H})=\sum_{i=1}^n\dim_k(H_i).$$
\end{theorem}

\begin{proof}
Let $k\in\{1,..., k'\}$ and let $S_i\subseteq V_i$ be a $k$-metric basis of $H_i$.   We will show that $S=\bigcup_{i=1}^{n}S_i$ is a $k$-metric generator for $G\odot \mathcal{H}$. Let us consider two different vertices $x,y$ of $G\odot \mathcal{H}$. We have the following cases.
\\
\\
\noindent {Case} 1: $x,y\in V_i$. Since $S_i$ is a $k$-metric basis of $H_i$, we have that $|\mathcal{D}_{H_i}(x,y)\cap S_i|\ge k$. Also, if $D(H_i)\le 2$, then for every $a,b\in V_i$, we have that $d_{H_i}(a,b)=d_{G\odot \mathcal{H}}(a,b)$. Now, since no vertex $u\in V(G\odot \mathcal{H})-V_i$ distinguishes the pair $x,y$, we conclude that $\mathcal{D}_{H_i}(x,y)=\mathcal{D}_{G\odot \mathcal{H}}(x,y)$. Thus, we obtain that $|\mathcal{D}_{G\odot \mathcal{H}}(x,y)\cap S|\ge k$.
\\
\\
\noindent {Case} 2: $x\in V_i$ and $y\in V_j$, $i\ne j$. For every $v\in S_i$  we have $d(x,v)\le 2<3\le d(y,v)$. Since $|S_i|\ge k$, we conclude that $|\mathcal{D}_{G\odot \mathcal{H}}(x,y)\cap S|\ge k$.
\\
\\
\noindent {Case} 3: $x,y\in V$. Assume $x=v_i$. Hence, for every $v\in S_i$, we have $d(x,v)=1<d(y,x)+1=d(y,v)$. Again, as $|S_i|\ge k$, we obtain that $|\mathcal{D}_{G\odot \mathcal{H}}(x,y)\cap S|\ge k$.
\\
\\
\noindent {Case} 4: $x\in V_i$ and $y\in V$. If $y=v_i$, then for every $v\in S_j$, with $j\ne i$, it follows that $d(x,v)=1+d(y,v)>d(y,v)$. Now,  if $y=v_l$, $l\ne i$, then for every $v\in S_l$, we have $d(x,v)=d(x,y)+d(y,v)>d(y,v)$. Finally, since $|S_j|\ge k$ and $|S_l|\ge k$, both possibilities lead to $|\mathcal{D}_{G\odot \mathcal{H}}(x,y)\cap S|\ge k$.

Thus, for every pair of different vertices $x,y\in V(G\odot \mathcal{H})$, we have that $|\mathcal{D}_{G\odot \mathcal{H}}(x,y)\cap S|\ge k$. So, $S$ is a $k$-metric generator for $G\odot \mathcal{H}$ and, as a consequence, $\dim_k(G\odot \mathcal{H})\le |S|=\sum_{i=1}^n\dim_k(H_i)$. The proof is completed by the lower bound of Theorem \ref{theoBoundCoronaDimH}.
\end{proof}

We must point out that Theorems \ref{theoBoundCoronaDimH} and \ref{theoEqualCoronaDimH} are generalizations of previous results established in \cite{Yero2011} for the case $k=1$.

Notice that there are values for $\dim_k(G\odot \mathcal{H})$ non achieving the bounds given in Theorem \ref{theoBoundCoronaDimH}. For instance,
if there exists a $k$-metric basis $S$ of $G\odot \mathcal{H}$ and a graph $H_i\in\mathcal{H}$ such that $\dim_k(H_i)<|S\cap V_i|<|V_i|$, then by Lemma \ref{lemmaKCorona} (iii) and (iv) we conclude $$\sum_{i=1}^{n}\dim_k(H_i)<\dim_k(G\odot \mathcal{H})<\sum_{i=1}^n|V_i|.$$
The results given in Proposition \ref{role-fan} show some examples for the observation above.

\subsection{The $k$-metric dimension of $K_1+H$ and its role in the study of the $k$-metric dimension of $G\odot \mathcal{H}$}

\begin{remark}\label{remark-H-K_1+H}
Let $H$ be a non-trivial graph.  If $B$ is a $k$-metric basis of $K_1+H$, then $B\cap V(H)$ is a $k$-metric generator for $H$.
\end{remark}

\begin{proof}
Let $B$ be a $k$-metric basis of $K_1+H$. Since the vertex of $K_1$ is adjacent to every vertex of $H$, for every $x,y\in V(H)$, we have $\vert B\cap \left( N_H(x) \bigtriangledown N_H(y)\cup \{x,y\}\right)\vert\ge k$ and, as a consequence,  $\vert B\cap \mathcal{D}_{H}(x,y)\vert \ge k$. Therefore,    $B\cap V(H)$ is a $k$-metric generator for $H$.
\end{proof}

\begin{corollary}
Let $H$ be a non-trivial graph.  If $K_1+H$ is a $k'$-metric dimensional graph, then   for every  $k\in\{1,\ldots,k'\}$,
 $$\dim_k(H)\le \dim_k(K_1+H).$$
\end{corollary}

Given a $k'$-metric dimensional graph $K_1+H$ and an integer $k\in\{1,\ldots,k'\}$, we define the following binary function.
$$f(H,k)=\left\{\begin{array}{ll}
0 & \textrm{if the vertex of $K_1$ does not belong to any $k$-metric basis of $K_1+H$,}\\
  & \\
1 & \textrm{if there exists a $k$-metric basis $S$ of $K_1+H$ containing the vertex of $K_1$.}
\end{array}\right.$$

\begin{theorem}\label{theoG_K1HBound}
Let $G$ be a connected graph of order $n\ge 2$ and let $\mathcal{H}$ be a family of $n$ non-trivial graphs such that for every $H_i\in \mathcal{H}$, the graph $K_1+H_i$ is $k_i$-metric dimensional. If $k'=\displaystyle\min_{i\in \{1,...,n\}}\{k_i\}$, then for any $k\in \{1,...,k'\}$,
$$\dim_k(G\odot \mathcal{H})\le \sum_{i=1}^n \left(\dim_k(K_1+H_i)-f(H_i,k)\right).$$
\end{theorem}

\begin{proof}
Let $V(G)=\{v_1,v_2,...,v_n\}$ be the vertex set of $G$. Now, for every $v_i\in V(G)$, let $B_i$ be a $k$-metric basis of $\langle v_i\rangle+ H_i$ containing $v_i$ if possible. Let $B'_i=B_i-\{v_i\}$ (notice that if for some $l\in \{1,...,n\}$, the vertex $v_l$ does not belong to any $k$-metric basis of $\langle v_l\rangle +H_l$,  then $B'_l=B_l$). From Remark \ref{remark-H-K_1+H}, we have that $B'_i$ is a $k$-metric generator for $H_i$. Thus, $|B'_i|\ge k$. We will show that $B=\bigcup_{i=1}^{n}B'_i$ is a $k$-metric generator for $G\odot \mathcal{H}$. We consider the following cases for any pair of  different vertices $x,y\in V(G\odot \mathcal{H})$.
\\
\\
\noindent {Case} 1: $x,y\in V_i$. Since no vertex outside of $V_i$ distinguishes $x,y$, we have that $|B_i'\cap \mathcal{D}_{G\odot \mathcal{H}}(x,y)|=|B'_i|\ge k$ and, consequently, $|B\cap \mathcal{D}_{G\odot \mathcal{H}}(x,y)|\ge k$.
\\
\\
\noindent {Case} 2: $x\in V_i$ and $y\in V_j$, $i\ne j$. For every $v\in B_i'$, we have that $d_{G\odot \mathcal{H}}(x,v)\le 2<3\le d_{G\odot \mathcal{H}}(y,v)$. Thus,  $B_i'\subset \mathcal{D}_{G\odot \mathcal{H}}(x,y)$ and, since $|B_i'|\ge k$, we conclude that $|B\cap \mathcal{D}_{G\odot \mathcal{H}}(x,y)|\ge k$.
\\
\\
\noindent {Case} 3: $x,y\in V$. Suppose now that $x=v_i$. In this case for every $v\in B_i'$ we have that $d_{G\odot \mathcal{H}}(x,v)=1<d_{G\odot \mathcal{H}}(y,x)+1=d_{G\odot \mathcal{H}}(y,v)$. Hence,  $B_i'\subset \mathcal{D}_{G\odot \mathcal{H}}(x,y)$ and, since $|B_i'|\ge k$, we conclude that $|B\cap \mathcal{D}_{G\odot \mathcal{H}}(x,y)|\ge k$.
\\
\\
\noindent {Case} 4: $x\in V_i$ and $y\in V$. If $y=v_i$, then for every  $v\in B_j'$, with $j\ne i$, we have $d_{G\odot \mathcal{H}}(x,v)=1+d_{G\odot \mathcal{H}}(y,v)>d_{G\odot \mathcal{H}}(y,v)$. Thus, $B_j'\subset \mathcal{D}_{G\odot \mathcal{H}}(x,y)$ and, since $|B_j'|\ge k$, we conclude that $|B\cap \mathcal{D}_{G\odot \mathcal{H}}(x,y)|\ge k$. Now, let us assume that $y=v_j$, with $j\ne i$. In this case for every $v\in B_j'$ we have that $d_{G\odot \mathcal{H}}(x,v)=d_{G\odot \mathcal{H}}(x,y)+d_{G\odot \mathcal{H}}(y,v)>d_{G\odot \mathcal{H}}(y,v)$. So, $B_j'\subset \mathcal{D}_{G\odot \mathcal{H}}(x,y)$ and, as $|B_j'|\ge k$, we conclude that $|B\cap \mathcal{D}_{G\odot \mathcal{H}}(x,y)|\ge k$.

Therefore, $B$ is a $k$-metric generator for $G\odot H$ and, as a consequence,
$$\dim_k(G\odot H)\le \vert B\vert= \sum_{i=1}^n|B_i'|=\sum_{i=1}^n \left(\dim_k(\langle v_i\rangle+H_i)-f(H_i,k)\right).$$
Since $\langle v_i\rangle+H_i\cong K_1+H_i$, the proof is complete.
\end{proof}

To see that the equality in Theorem \ref{theoG_K1HBound} is attained, we take a family $\mathcal{H}$ such that for every $H_i\in \mathcal{H}$ the graph $K_1+H_i$ is  $k$-metric dimensional and $\dim_k(H_i)=\vert V_i \vert$.  In such a situation, since the vertex of $K_1$ does not distinguish any pair of vertices belonging to $V_i$, we have that either  $\dim_k(K_1+H_i)=\vert V_i \vert$, in which case the vertex of $K_1$ does not belong to any $k$-metric basis of $K_1+H_i$, or $\dim_k(K_1+H_i)=\vert V_i \vert+1$,  in which case the vertex of $K_1$  belongs to any $k$-metric basis of $K_1+H_i$. Thus, Theorem \ref{theoG_K1HBound} leads to $ \dim_k(G\odot \mathcal{H})\le \sum_{i=1}^n \vert V_i\vert.$ As shown in Corollary \ref{PropoTheUpperboundTight}, the equality is attained.
For instance, we can take $k=2$ and every $H_i=K_r$, where $r\ge 2$, or $k=3$ and every $H_i=P_4$, or $k=4$ and every $H_i=C_5$.

Since for any $x,y\in V(H)$ it holds $N_{H}(x)\triangledown N_{H}(y)=N_{\bar{H}}(x)\triangledown N_{\bar{H}}(y)$, where $\bar{H}$ denotes the complement of graph $H$, we deduce that $(N_{H}(x)\triangledown N_{H}(y))\cup\{x,y\}=(N_{\bar{H}}(x)\triangledown N_{\bar{H}}(y))\cup\{x,y\}$.  Therefore, the next result is deduced.

\begin{lemma}\label{lemComplement}
Let $H$ be a non-trivial graph such that the vertex of $K_1$ does not belong to any $k$-metric basis of $K_1+H$. Any $k$-metric basis of $K_1+H$ is $k$-metric basis of $K_1+\bar{H}$ and, therefore $\dim_k(K_1+H)=\dim_k(K_1+\bar{H})$.
\end{lemma}

By Corollary \ref{dimensional-fan-wheel}  the wheel graph $K_1+C_r=W_{1,r}$ is $4$-metric dimensional for $r\ge 7$. Therefore, the next lemma makes only sense for $k\le 4$. We do not consider the case $k=1$, since it has been previously studied in \cite{Buczkowski2003}.

\begin{lemma}\label{lemmaW1n-K_Plus_2}
Let $C_r$ be a cycle graph of order $r\ge 7$, and let $k\in \{2,3,4\}$. If there exists $S\subseteq V(C_r)$ such that $|\mathcal{D}_{W_{1,r}}(x,y)\cap S|\ge k$ for every $x,y\in V(C_r)$, then $|S|\ge k+2$.
\end{lemma}

\begin{proof}
Let $V(C_r)=\{u_0,u_2,...,u_{r-1}\}$ be the vertex set of the cycle $C_r$. The subscripts of $u_i\in V(C_r)$ will be taken modulo $r$. Notice that $\mathcal{D}_{W_{1,r}}(u_i,u_{i+1})=\{u_{i-1},u_i,u_{i+1},u_{i+2}\}$.

We first consider the case  $r\ge 8$. Since $\mathcal{D}_{W_{1,r}}(u_i,u_{i+1})\cap \mathcal{D}_{W_{1,r}}(u_{i+4},u_{i+5})=\emptyset$,  $|\mathcal{D}_{W_{1,r}}(u_i,u_{i+1})\cap S|\ge k$ and $|\mathcal{D}_{W_{1,r}}(u_{i+4},u_{i+5})\cap S|\ge k$, we deduce that $|S|\ge 2k$. Thus, for $k\ge 2$ we have that $|S|\ge k+2$.

We now consider the case $r=7$. Since $\mathcal{D}_{W_{1,r}}(u_i,u_{i+1})\cap \mathcal{D}_{W_{1,r}}(u_{i+4},u_{i+5})=\{u_{i+6}\}$, in this case we have  $|S|\ge 2k-1$. So for $k\in\{3,4\}$ it holds $|S|\ge k+2$. Now we take $k=2$. Suppose that  $|S|= 3$. If $S$ is composed by non-consecutive vertices, say $S=\{u_i,u_{i+2},u_{i+4}\}$, then $\vert \mathcal{D}_{W_{1,r}}(u_{i+4},u_{i+5})\cap S \vert=1$, which is a contradiction. If there are two consecutive vertices in $S$, say $u_i,u_{i+1}\in S$, then $\vert \mathcal{D}_{W_{1,r}}(u_{i+3},u_{i+4})\cap S \vert\le 1$, which is a contradiction. Hence, $|S|\ge 4$ and, as a consequence, for $k=2$ we have that $|S|\ge k+2$.
\end{proof}

In order to present our next result, we need to introduce a new notation. Given a family of graphs $\mathcal{H}=\{H_1,\ldots,H_n\}$, we define $\mathcal{\overline{H}}$ as the family of the complement graphs of each $H_i\in\mathcal{H}$, \textit{i.e.}, $\mathcal{\overline{H}}=\{\overline{H}_1,\ldots,\overline{H}_n\}$.

\begin{theorem}\label{theoG_K1HEqual}
Let $G$ be a connected graph of order $n\ge 2$ and let $\mathcal{H}$ be a family of $n$ connected non-trivial graphs. If for every $H_i\in\mathcal{H}$ it holds  $D(H_i)\ge 6$ or $H_i$ is a cycle graph of order greater than or equal to seven, then for any $k\in\{1,\ldots,\mathcal{C}(\mathcal{H})\}$,
 $$\dim_k(G\odot\mathcal{H})=\dim_k(G\odot\mathcal{\overline{H}})=\sum_{i=1}^n\dim_k(K_1+ H_i).$$
\end{theorem}

\begin{proof}
The case $k=1$, where every $H_i$ is isomorphic to a fixed graph $H$, was studied in \cite{Yero2011}. Moreover, the procedure to prove the case when $k=1$ and $\mathcal{H}$ contains at least two non-isomorphic graphs, is quite similar to the one presented in \cite{Yero2011}. Hence, from now on we assume that $k\ge 2$.

By Remark \ref{Comparakenproductoyk1+H}, if for every $H_i\in\mathcal{H}$ of order $n_i$, the graph $K_1+H_i$ is $k_i$-metric dimensional, then for $k\in\{1,\ldots,\min_{i\in\{1,\ldots,n\}}\{k_i\}\}$ there exist $k$-metric bases of $G\odot\mathcal{H}$.  By Lemma \ref{K1_H_is_CH_dimensional} and Proposition \ref{K1-not-belong}, we deduce that $\mathcal{C}(\mathcal{H})=\min_{i\in\{1,\ldots,n\}}\{k_i\}$.

Let $S$ be a $k$-metric basis of $G\odot\mathcal{H}$. We will show that $S_i=S\cap V_i$ is a $k$-metric generator for $\langle v_i\rangle+H_i$. Notice that by Lemma \ref{lemmaKCorona}, for every $x,y\in V_i$ we have that $|S_i\cap\mathcal{D}_{\langle v_i\rangle+H_i}(x,y)|=|S_i\cap\mathcal{D}_{G\odot\mathcal{H}}(x,y)|\ge k$. Now we differentiate two cases in order to show that for every $x\in V_i$  it holds $|S_i\cap\mathcal{D}_{\langle v_i\rangle+H_i}(x,v_i)|\ge k$.
\\
\\
\noindent {Case} 1: $H_i$ is a cycle graph of order $n'\ge 7$. Since $n'\ge 7$, by Lemma \ref{lemmaW1n-K_Plus_2}, we have that $|S_i|\ge k+2$. Notice that for any $x\in V_i$ there exist exactly two vertices $y,z\in V_i$ such that $d_{H_i}(x,y)=d_{H_i}(x,z)=1$. Since $|S_i|\ge k+2$, for every $x\in V_i$ we have that there exist at least $k$ elements $u$ of $S_i$ such that $d_{H_i}(u,x)>1$, and as a consequence, $d_{\langle v_i\rangle+H_i}(u,x)=2$. Hence, $|S_i\cap\mathcal{D}_{\langle v_i\rangle+H_i}(x,v_i)|\ge k$.
\\
\\
\noindent {Case} 2: $D(H_i)\ge 6$. If for every $x\in V_i$ there exist at least $k$ elements in $S_i$ which are not adjacent to $x$, then the result holds. Hence, given $z\in V_i$, we define $R_i(z)=\left(V_i-N_{H_i}(z)\right)\cap S_i$. Suppose that there exists $x\in V_i$ such that $0\le |R_i(x)|\le k-1$.

Now, let $F_i(x)=S_i-R_i(x)$. Since $|S_i|\ge k$, we have that $F_i(x)\ne\emptyset$. If $V_i=F_i(x)\cup \{x\}$, then $D(H_i)\le 2$, which is a contradiction.
Now, if for every $y\in V_i-\left( F_i(x)\cup\{x\}\right)$ there exists $z\in F_i(x)$ such that $d_{H_i}(y,z)=1$, then $D(H_i)\le 4$, which is a contradiction. So, we assume that there exists a vertex $y\in V_i-\left(F_i(x)\cup\{x\}\right)$ such that $d_{H_i}(y,z)>1$, for every $z\in F_i(x)$. If $V_i=F_i(x)\cup\{x,y\}$, then $y\sim x$ and, as a consequence, $D(H_i)=2$, which is also a contradiction. Hence, $V_i-(F_i(x)\cup\{x,y\})\ne\emptyset$.

Since $N_{H_i}(y)\cap F_i(x)=\emptyset$ and  $|R_i(x)|<k$,  and also for any $w\in V_i-(F_i(x)\cup\{x,y\})$ we have that $\mathcal{D}_{G\odot\mathcal{H}}(y,w)=\left(N_{H_i}(y)\triangledown N_{H_i}(w)\right)\cup \{y,w\}$ and $|\mathcal{D}_{G\odot\mathcal{H}}(y,w)\cap S_i|\ge k$, we deduce that $N_{H_i}(w)\cap F_i(x)\ne \emptyset$, and this leads to $D(H_i)\le 5$, which is also a contradiction.

Therefore, if $D(H_i)\ge 6$, then for every $x\in V_i$ we have that $|R_i(x)|\ge k$ and, as a consequence,  for every $x\in V_i$ there exist at least $k$ vertices  $u\in S_i$ such that $d_{\langle v_i\rangle+H_i}(u,x)=2$. Hence, $|S_i\cap\mathcal{D}_{\langle v_i\rangle+H_i}(x,v_i)|\ge k$.

We have shown that $S_i$ is a $k$-metric generator for $\langle v_i\rangle+H_i$ and, as a consequence, $\dim_k(\langle v_i\rangle+ H_i)\le \vert S_i\vert$. Now, by Lemma \ref{lemmaKCorona} (iii) we have that $V(G)\cap S=\emptyset$ and, consequently, $S=\bigcup_{i=1}^n S_i$. Therefore, $$\dim_k(G\odot\mathcal{H})=|S|=\sum_{i=1}^n |S_i|\ge\sum_{i=1}^n\dim_k(K_1+ H_i).$$
Finally, by Theorem \ref{theoG_K1HBound} and Lemma \ref{lemComplement}, the proof is completed.
\end{proof}

By Theorems \ref{theoG_K1HBound} and \ref{theoG_K1HEqual} we deduce the following result.

\begin{proposition}\label{K1-not-belong}
Let $H$ be a connected graph such that $K_1+H$ is $k'$-metric dimensional and let $k\in \{1,...,k'\}$. If $D(H)\ge 6$ or $H$ is a cycle graph of order greater than or equal seven, the vertex of $K_1$ does not belong to any $k$-metric basis of $K_1+H$.
\end{proposition}

In order to present our next result we introduce a new definition. Given a family of $n$ graphs $\mathcal{H}$, we denote by $K_1\diamond\mathcal{H}$ the family of graphs formed by the graphs $K_1+H_i$ for every $H_i\in\mathcal{H}$, \textit{i.e.}, $K_1\diamond\mathcal{H}=\{K_1+H_1,K_1+H_2,\ldots,K_1+H_n\}$.

\begin{proposition}\label{propEqH-K1_H}
Let $G$ be a connected graph of order $n\ge 2$, let $\mathcal{H}$ be a family of $n$ connected graphs, and let $K_1+H_i$ be a $k_i$-metric dimensional graph for every $H_i\in\mathcal{H}$. If for every $H_i\in\mathcal{H}$ holds that $D(H_i)\ge 6$ or $H_i$ is a cycle graph of order greater than or equal to seven, then for any $k\in\{1,\ldots,\mathcal{C}(\mathcal{H})\}$, $$\dim_k(G\odot\mathcal{H})=\dim_k(G\odot\mathcal{\overline{H}})=\dim_k\left(G\odot (K_1\diamond\mathcal{H})\right).$$
\end{proposition}

\begin{proof}
Since for every $H_i\in\mathcal{H}$, it follows $D(K_1+H_i)=2$, by Theorem \ref{theoEqualCoronaDimH}, $\dim_k\left(G\odot (K_1\diamond\mathcal{H})\right)= \sum_{i=1}^n\dim_k(K_1+ H_i)$. Also, by Theorem \ref{theoG_K1HEqual}, $\dim_k(G\odot\mathcal{H})=\dim_k(G\odot\mathcal{\overline{H}})=\sum_{i=1}^n\dim_k(K_1+ H_i)$. So, the result follows.
\end{proof}

Next we consider some special classes of graphs of the form $K_1+H$, the so called fan graphs and wheel graphs.

\subsubsection{The particular case of fan graphs and wheel graphs }

In order to study the $k$-metric dimension of fan graphs, we will use the following notation.  Let $V(P_n)=\{u_1,u_2,...,u_n\}$ be the vertex set of the path $P_n$ and let $F_{1,n}=\langle u\rangle+P_n$. We assume that $u_i\sim u_{i+1}$  for each $i\in \{ 1,...,n-1\}$.

By Corollary \ref{dimensional-fan-wheel} we know that the fan graphs $F_{1,n}$, $n\ge 4$, are $3$-metric dimensional, so $\dim_k(F_{1,n})$ makes sense for $k\in \{1,2,3\}$.
In this section we study the cases $k=2$ and $k=3$, since the case $k=1$  was previously studied in \cite{Hernando2005}, that is:

$$\dim_1(F_{1,n})=\left\{\begin{array}{ll}
1, & \textrm{if $n=1$,}\\
2, & \textrm{if $n=2,3$,}\\
3, & \textrm{if $n=6$,}\\
{\left\lfloor{\frac{2n+2}{5}}\right\rfloor}, & \textrm{otherwise.}\\
\end{array}\right.
$$

We first present some useful lemmas.

\begin{lemma}\label{lemma-Fan-mayor-2k}
Let $k\in \{2,3\}$ and let $n\ge 6$ be an integer. For any $k$-metric basis $S$ of $F_{1,n}$ it holds $|S\cap V(P_n)|\ge 2k$.
\end{lemma}

\begin{proof}
Notice that $\mathcal{D}_{F_{1,n}}(u_1,u_2)=\{u_1,u_2,u_3\}$ and $\mathcal{D}_{F_{1,n}}(u_{n-1},u_n)= \{u_{n-2},u_{n-1},u_n\}$. Since $S$ is a $k$-metric basis of $F_{1,n}$, we have $|S\cap \mathcal{D}_{F_{1,n}}(u_1,u_2)|\ge k$ and $|S\cap \mathcal{D}_{F_{1,n}}(u_{n-1},u_n)|\ge k$. As $n\ge 6$, it holds  $\mathcal{D}_{F_{1,n}}(u_1,u_2)\cap \mathcal{D}_{F_{1,n}}(u_{n-1},u_n)=\emptyset$. Therefore, $|S\cap V(P_n)|\ge 2k$.
\end{proof}

\begin{lemma}\label{No-Belong-CenterK1H_k}
Let $H$ be a non-trivial graph, let $K_1 + H$ be a $k'$-metric dimensional graph, and let $k\in \{1,\ldots,k'\}$. If for every $k$-metric basis $S$ of $K_1 + H$ we have that $|S\cap V(H)|\ge k+\Delta(H)$, then the vertex of $K_1$ does not belong to any $k$-metric basis of $K_1 + H$.
\end{lemma}

\begin{proof}
Let $v$ be the vertex of $K_1$ and let $S$ be a $k$-metric basis of $K_1+H$. We will show that $S'=S-\{v\}$ is a $k$-metric generator for $K_1+H$.

On one hand, for every $x\in V(H)$ we have $|S'\cap \mathcal{D}_{K_1+H}(x,v) |=| S'\cap (V(H)-N_{H}(x)) |\ge k$, as  $|S'\cap V(H)|=|S\cap V(H)|\ge k+\Delta(H)$.

On the other hand, for any $x,y\in V(H)$ we have $|S'\cap \mathcal{D}_{K_1+H}(x,y)|=|S\cap \mathcal{D}_{K_1+H}(x,y)|\ge k$, as $v\not\in \mathcal{D}_{K_1+H}(x,y)$.

Therefore, $S'$ is a $k$-metric generator for $K_1+H$ and, by the minimality of $S$, the set $S'$ is a $k$ metric basis of $K_1+H$.
\end{proof}

By performing some simple calculations, we have observed that $\dim_2(F_{1,2})=3$, $\dim_2(F_{1,3})=4$, $\dim_2(F_{1,4})=\dim_2(F_{1,5})=4$ and $\dim_3(F_{1,4})=\dim_3(F_{1,5})=5$. The remaining values of $\dim_k(F_{1,n})$ are obtained in our next proposition.

\begin{proposition}\label{value-fans}
For any integer $n\ge 6$,
\begin{enumerate}[{\rm (i)}]
\item $\dim_2(F_{1,n})=\left\lceil(n+1)/2\right\rceil$.
\item $\dim_3(F_{1,n})= n-\left\lfloor(n-4)/5\right\rfloor$
\end{enumerate}
\end{proposition}

\begin{proof}

\noindent (i) We shall prove that $A=\{u_i\in V(P_n)\,:\,i\equiv 1\;(2)\}\cup \{u_n\}$  is a $2$-metric generator for $F_{1,n}$. Let $x,y$ be two different vertices of $F_{1,n}=\langle u\rangle+P_n$.

If  $x=u$, then $d_{F_{1,n}}(x,u_i)=1$ for every $u_i\in V(P_n)$.
Since $|A|\ge 4$ and there exist at most two vertices $u_j,u_l\in V(P_n)$ such that $d_{F_{1,n}}(y,u_j)=d_{F_{1,n}}(y,u_l)=1$, we have  $|\mathcal{D}_{F_{1,n}}(u,y)\cap A|\ge 2$.

Let us now assume that $x,y\in V(P_n)$. If $x,y\in A$, then they are distinguished by themselves and, if $x,y\notin A$, then there exist at least two vertices $u_i,u_j\in A$ such that $u_i,u_j\in  N(x) \bigtriangledown N(y)\subset \mathcal{D}_{F_{1,n}}(x,y)$. Finally, if $x\in A$ and $y\notin A$, then there exists a vertex $u_l\in A-\{x\}$
 such that $u_l\in N(y) - N(x)$. Therefore, $A$ is a $2$-metric generator for $F_{1,n}$ and, as a consequence, $\dim_2(F_{1,n})\le |A|=\left\lceil(n+1)/2\right\rceil$.

It remains to show that $\dim_2(F_{1,n})\ge  \left\lceil(n+1)/2\right\rceil$.  With this aim, we take an arbitrary  $k$-metric basis $A'$  of $F_{1,n}$. Since $n\ge 6$, by Lemmas \ref{lemma-Fan-mayor-2k} and \ref{No-Belong-CenterK1H_k},  $u\not\in A'$. Notice that $\mathcal{D}_{F_{1,n}}(u_1,u_2)=\{u_1,u_2,u_3\}$ and $\mathcal{D}_{F_{1,n}}(u_{n-1},u_n)=\{u_{n-2},u_{n-1},u_n\}$. Thus, $ |A'\cap \{u_1,u_2,u_3\}| \ge 2$  and $ |A'\cap \{u_{n-2},u_{n-1},u_n\}| \ge 2$.
So, for $n=6$, then $|A'|\ge 4$ and we are done. From now on we consider $n\ge 7$.
 Let $M(P_n)=V(P_n)-\{u_1,u_2,u_3,u_{n-2},u_{n-1},u_{n}\}$. Assume for purposes of contradiction that $|A'\cap M(P_n)|\le\left\lfloor(n-6)/2\right\rfloor-1$. We consider the following cases.

\begin{enumerate}[(1)]
\item $n-6=4p$ or $n-6=4p+1$ for some positive integer $p$. Let $Q_i=\{u_{4i},u_{4i+1},u_{4i+2},u_{4i+3}\}$, $1\le i\le p$. Notice that every $Q_i\subset M(P_n)$. Since $|A'\cap M(P_n)|<\left\lfloor(n-6)/2\right\rfloor=2p$, there exists at least a set $Q_j=\{u_{4j},u_{4j+1},u_{4j+2},u_{4j+3}\}$ such that $|Q_j\cap A'|\le 1$. Since $\mathcal{D}_{F_{1,n}}(u_{4j+1},u_{4j+2})=\{u_{4j},u_{4j+1},u_{4j+2},u_{4j+3}\}$, we deduce that $u_{4j+1},u_{4j+2}$ are distinguished by at most one vertex of $A'$, which is a contradiction.
\item $n-6=4p+2$ for some positive integer $p$.  As above, let $Q_i=\{u_{4i},u_{4i+1},u_{4i+2},u_{4i+3}\}$, $1\le i\le p$. Notice that $M(P_n)=\left(\bigcup_{i=1}^{p} Q_i\right)\cup \{u_{4(p+1)},u_{4(p+1)+1}\}$. If there exists at least one $Q_i$ such that $|Q_i\cap A'|\le 1$, then we have a contradiction as in the case above. Thus, $|Q_i\cap A'|\ge 2$ for all $1\le i\le p$ and we have
$$2p=\left\lfloor(n-6)/2\right\rfloor-1\ge|A'\cap M(P_n)|=\sum_{i=1}^{p}|Q_i\cap A'|+|A'\cap \{u_{4(p+1)},u_{4(p+1)+1}\}|\ge 2p.$$
As a consequence, it follows $|Q_j\cap A'|=2$ for every $j\in \{1,...,p\}$ and $A'\cap \{u_{4(p+1)},u_{4(p+1)+1}\}=\emptyset$. Now, if $u_{4p+2},u_{4p+3}\in A'$, then $u_{4p},u_{4p+1}\notin A'$. Thus, $u_{4p+1},u_{4p+3}$ are distinguished only by $u_{4p+3}$, which is a contradiction. Conversely, if $u_{4p+2}\notin A'$ or $u_{4p+3}\notin A'$, then $|A'\cap \{u_{4p+2},u_{4p+3},u_{4(p+1)},u_{4(p+1)+1}|\le 1$ and, like in the previous case, we obtain that $u_{4p+3},u_{4(p+1)}$ are distinguished by at most one vertex, which is also a contradiction.
\item If $n-6=4p+3$, then we obtain a contradiction by proceeding analogously to Case $2$ ($n-6=4p+2$).
\end{enumerate}

Thus, $|A'\cap M(P_n)|\ge\left\lfloor(n-6)/2\right\rfloor$ and we obtain that
$\dim_2(F_{1,n})=|A'|=|A'\cap M(P_n)|+|A'\cap \mathcal{D}_{F_{1,n}}(u_1,u_2)| + |A'\cap \mathcal{D}_{F_{1,n}}(u_{n-1},u_n)|\ge \left\lfloor(n-6)/2\right\rfloor+4=\left\lceil(n+1)/2\right\rceil$. Therefore, (i) follows.
\\
\\
\noindent (ii) Let $S=V(P_n)-\{u_i\in V(P_n):\,i\equiv 0\;(5)\; \wedge \; 1\le i\le n-4\}$. Notice that $|S|=n-\left\lfloor(n-4)/5\right\rfloor$. We claim that $S$ is a $3$-metric generator for $F_{1,n}$. Let $x,y$ be two different vertices of $F_{1,n}$.

If  $x=u$,  then $d_{F_{1,n}}(x,u_i)=1$ for every $u_i\in V(P_n)$. Also, there exist at most two vertices $u_j,u_l\in V(P_n)$ such that $d_{F_{1,n}}(y,u_j)=d_{F_{1,n}}(y,u_l)=1$. Since $|S|\ge 6$ the vertices $x,y$ are distinguished by at least three vertices of $S$.

Now suppose $x,y\in V(P_n)$. According to the construction of $S$, there exist at least three different vertices $u_{i_1},u_{i_2},u_{i_3}\in S$ such that $d_{F_{1,n}}(x,u_{i_j})\ne d_{F_{1,n}}(y,u_{i_j})$, with $j\in\{1,2,3\}$. (Notice that $x$ or $y$ could be equal to some $u_{i_j}$, $j\in\{1,2,3\}$)

Thus, $S$ is a $3$-metric generator for $F_{1,n}$ and, as a result,  $\dim_3(F_{1,n})\le |S|=n-\left\lfloor(n-4)/5\right\rfloor$.

It remains to show that $\dim_3(F_{1,n})\ge n-\left\lfloor(n-4)/5\right\rfloor$.
Now, let $S'$ be a $3$-metric basis of $F_{1,n}$. Since $n\ge 6$, by Lemmas \ref{lemma-Fan-mayor-2k} and \ref{No-Belong-CenterK1H_k}, $u\not \in S'$. Also, notice that two adjacent vertices $u_{i},u_{i+1}$ are distinguished by themselves and at least one neighbor $u_{i-1}$ or $u_{i+2}$. So, at least three of them belong to $S'$. Now, if there exist three consecutive vertices $u_{i-1},u_i,u_{i+1}\in S'$ such that $u_{i-2},u_{i+2}\notin S'$, then the vertices $u_{i-1},u_{i+1}$ are not distinguished by at least three vertices of $S'$, which is a contradiction. Thus, if two vertices $u_i,u_j\notin S'$, then $i-j\equiv 0\;(5)$ and, as a consequence, per each five consecutive vertices of $V(P_n)$, at least four of them  are in $S'$, or equivalently, at most one does not belong to $S'$. Moreover, notice that $\mathcal{D}_{F_{1,n}}(u_1,u_2)=\{u_1,u_2,u_3\}$, $\mathcal{D}_{F_{1,n}}(u_1,u_3)=\{u_1,u_3,u_4\}$, $\mathcal{D}_{F_{1,n}}(u_{n-1},u_{n})=\{u_{n-2},u_{n-1},u_{n}\}$ and $\mathcal{D}_{F_{1,n}}(u_{n-2},u_{n})=\{u_{n-3},u_{n-2},u_{n}\}$. By Lemma \ref{lemmaBelongKBasis}, $\{u_1,u_2,u_3,u_4,u_{n-3},u_{n-2},u_{n-1},u_n\}\subset S'$. Hence, $|\overline{S'}|\le \left\lfloor(n-4)/5\right\rfloor+1$, where we refer to $\overline{S'}$ as the complement of the set $S'$. Finally, we have that $\dim_3(F_{1,n})=|S'|=n+1-|\overline{S'}|\ge n-\left\lfloor(n-4)/5\right\rfloor$. Therefore, $\dim_3(F_{1,n})= n-\left\lfloor(n-4)/5\right\rfloor$.
\end{proof}

The next result shows the relationship between $\dim_k(G\odot\mathcal{H})$ and $\dim_k(G\odot(K_1\diamond\mathcal{H}))$ for a family $\mathcal{H}$ of paths of order greater than five and $k\in \{1,2,3\}$. We only consider $k\in \{1,2,3\}$, since for $n'\ge 6$ we have that $\mathcal{C}(P_{n'})=\mathcal{C}(F_{1,n'})=3$, and as a consequence, by Theorem \ref{theoKmetricCorona}, $G\odot\mathcal{H}$ and $G\odot(K_1\diamond\mathcal{H})$ are $3$-metric dimensional.

\begin{proposition}\label{role-fan}
Let $G$ be a connected graph of order $n\ge 2$ and let $\mathcal{H}$ be a family of paths. If every path $P_i\in\mathcal{H}$ has order $n_i$, then the following statements hold.
\begin{enumerate}[{\rm (i)}]
\item If $n_i\ge 7$ for $i\in\{1,\ldots,n\}$, then $\dim(G\odot\mathcal{H})=\dim(G\odot\mathcal{\overline{H}})=\dim(G\odot(K_1\diamond\mathcal{H}))=\sum_{i=1}^n\left\lfloor(2n_i+2)/5\right\rfloor$.
\item If $n_i\ge 6$ for $i\in\{1,\ldots,n\}$, then $\dim_2(G\odot\mathcal{H})=\dim_2(G\odot\mathcal{\overline{H}})=\dim_2(G\odot(K_1\diamond\mathcal{H}))=\sum_{i=1}^n\left\lceil(n_i+1)/2\right\rceil$.
\item If $n_i\ge 6$ for $i\in\{1,\ldots,n\}$, then $\dim_3(G\odot\mathcal{H})=\dim_3(G\odot\mathcal{\overline{H}})=\dim_3(G\odot(K_1\diamond\mathcal{H}))=\sum_{i=1}^n\left(n_i-\left\lfloor(n_i-4)/5\right\rfloor\right)$.
\end{enumerate}
\end{proposition}

\begin{proof}
If $n_i\ge 7$, then by Theorem \ref{theoEqualCoronaDimH} and Propositions  \ref{propEqH-K1_H} and \ref{value-fans} the result follows. Hence, we only need to prove that $\dim_k(G\odot\mathcal{H})=\dim_k(G\odot(K_1\diamond\mathcal{H}))$ for the cases where $n_i=6$ and $k\in \{2,3\}$. We recall that, by Lemma \ref{lemma-Fan-mayor-2k}, for $k\in \{2,3\}$, $n'\ge 6$ and any $k$-metric basis $S$ of $F_{1,n'}$, it holds $|S\cap V(P_{n'})|\ge 2k$. Since for $k\in\{2,3\}$, we have that $|S|\ge k+2$. Thus, by a procedure analogous to  the one used in  the proof of Theorem \ref{theoG_K1HEqual}, Case $1$, we deduce that $\dim_k(G\odot\mathcal{H})=\sum_{i=1}^n\dim_k(F_{1,n_i})$. Since $F_{1,n_i}$ has diameter two, by Theorem \ref{theoEqualCoronaDimH}, $\dim_k(G\odot(K_1\diamond\mathcal{H}))= \sum_{i=1}^n\dim_k(F_{1,n_i})$. Therefore, by Proposition \ref{value-fans} the result follows.
\end{proof}

Let $V(C_n)=\{u_0,u_2,...,u_{n-1}\}$ be the vertex set of the cycle $C_n$ in $W_{1,n}=K_1+C_n$ and let $u$ be the central vertex of the wheel graph. From now on, all the operations with the subscripts of $u_i\in V(C_n)$ will be taken modulo $n$.

Since $W_{1,3}$ and $W_{1,4}$ have twin vertices, they are $2$-metric dimensional graphs. Also, by Corollary \ref{dimensional-fan-wheel} we know that the wheel graphs $W_{1,n}$, $n\ge 5$, are $4$-metric dimensional, \textit{i.e}, $\dim_k(W_{1,n})$ makes sense for $k\in \{1,2,3,4\}$. 
The case $k=1$ was previously studied in \cite{Buczkowski2003}, that is:

$$\dim_1(W_{1,n})=\left\{\begin{array}{ll}
3, & \textrm{if $n=3,6$,}\\
2, & \textrm{if $n=4,5$,}\\
{\left\lfloor{\frac{2n+2}{5}}\right\rfloor}, & \textrm{otherwise.}\\
\end{array}\right.
$$

We now study $\dim_k(W_{1,n})$ for $k\in \{2,3,4\}$.  We first give a useful lemma.

\begin{lemma}\label{Any-KGenerator-K1H_k}
Let $H$ be a non-trivial graph and let $K_1 + H$ be a $k'$-metric dimensional graph. Let $k\in \{1,\ldots,k'\}$ and $S\subseteq V(H)$. If for every $x,y\in V(H)$, $|S\cap\mathcal{D}_{K_1+H}(x,y)|\ge k$ and $|S|\ge k+\Delta(H)$, then $S$ is a $k$-metric generator for $K_1 + H$.
\end{lemma}

\begin{proof}
Let $v$ be the vertex of $K_1$. Since for every $x,y\in V(H)$ we have that $|S\cap\mathcal{D}_{K_1+H}(x,y)|\ge k$, in order to prove that $S$ is a $k$-metric generator for $K_1 + H$, it is enough proving that for every $x\in V(H)$ the condition $|\mathcal{D}_{K_1+H}(x,v)\cap S|\ge k$ is satisfied. Notice that for every $x\in V(H)$ we have that $\mathcal{D}_{K_1+H}(x,v)=\left(V(H)-N_{H}(x)\right)\cup \{v\}$. Since $|S|\ge k+\Delta(H)$, for every $x\in V(H)$ there exist $k$ vertices $y\in S\cap (V(H)-N_{H}(x))$. Thus, for every $x\in V(H)$ it holds that $|\mathcal{D}_{K_1+H}(x,v)\cap S|\ge k$. Therefore, $S$ is a $k$-metric generator for $K_1 + H$.
\end{proof}

By performing some simple calculations, we have that $\dim_2(W_{1,3})=\dim_2(W_{1,4})=\dim_2(W_{1,5})=\dim_2(W_{1,6})=4$, $\dim_3(W_{1,5})=\dim_3(W_{1,6})=5$ and $\dim_4(W_{1,5})=\dim_4(W_{1,6})=6$. Next we present a formula for the $k$-metric dimension of wheel graphs for $n\ge 7$ and $k\in \{2,3,4\}$.

\begin{proposition}\label{value-wheels}
For any $n\ge 7$,
\begin{enumerate}[{\rm (i)}]
\item $\dim_2(W_{1,n})=\left\lceil n/2\right\rceil$.
\item $\dim_3(W_{1,n})= n-\left\lfloor n/5\right\rfloor$.
\item $\dim_4(W_{1,n})=n$.
\end{enumerate}
\end{proposition}

\begin{proof}
Since $n\ge 7$, by Proposition \ref{K1-not-belong}, the central vertex of $W_{1,n}$ does not belong to any $k$-metric basis of $W_{1,n}$. Thus, any $k$-metric basis of $W_{1,n}$ is a subset of $V(C_n)$. Let $S_{k}\subset V(C_n)$, $k\in \{2,3,4\}$, be a set of vertices of $W_{1,n}$ such that $|S_{2}|<\left\lceil n/2\right\rceil$, $|S_{3}|<n-\left\lfloor n/5\right\rfloor$ and $|S_{4}|<n$. We claim that $S_{k}$ is not a $k$-metric generator for $W_{1,n}$ with $k\in \{2,3,4\}$. Consider each $S_{k}$ independently:
\\
\\
\noindent $k=2$. Since $|S_{2}|<\left\lceil n/2\right\rceil$, there exist four consecutive vertices $u_i,u_{i+1},u_{i+2},u_{i+3}$ such that at most one of them belongs to $S_{2}$. Thus, $|\mathcal{D}_{W_{1,n}}(u_{i+1},u_{i+2})\cap S_2|\le 1$.
\\
\\
\noindent  $k=3$. Since $|S_{3}|<n-\left\lfloor n/5\right\rfloor$, there exist five consecutive vertices $u_i,u_{i+1},u_{i+2},u_{i+3},u_{i+4}$ such that at most three of them belong to $S_{3}$. Thus, there exist four consecutive vertices $u_j,u_{j+1},u_{j+2},u_{j+3}\in \{u_i,u_{i+1},u_{i+2},u_{i+3},u_{i+4}\}$ such that at most two of them belong to $S_{3}$, with the exception of two cases. Hence, $|\mathcal{D}_{W_{1,n}}(u_{j+1},u_{j+2})\cap S_3|\le 2$. The two exceptional cases are when either $u_{i+1},u_{i+2},u_{i+3}\in S_{3}$ or $u_i,u_{i+2},u_{i+4}\in S_{3}$. In both cases, $|\mathcal{D}_{W_{1,n}}(u_{i+1},u_{i+3})\cap S_3|=2$.
\\
\\
\noindent $k=4$. Since $|S_{4}|<n$, there exist four consecutive vertices $u_i,u_{i+1},u_{i+2},u_{i+3}$ such that at most three of them belong to $S_{4}$. Thus, $|\mathcal{D}_{W_{1,n}}(u_{i+1},u_{i+2})\cap S_4|\le 3$.

Therefore, as we claimed, $S_{k}$ is not a $k$-metric generator for $W_{1,n}$, with $k\in \{2,3,4\}$ and so $\dim_2(W_{1,n})\ge\left\lceil n/2\right\rceil$, $\dim_3(W_{1,n})\ge n-\left\lfloor n/5\right\rfloor$ and $\dim_4(W_{1,n})\ge n$.

Since $n\ge 7$, by Proposition \ref{K1-not-belong}, the central vertex of $W_{1,n}$ does not belong to any $k$-metric basis of $W_{1,n}$. Thus, $V(C_n)$ is a $4$-metric generator for $W_{1,n}$ and, as a result, $\dim_4(W_{1,n})= n$. It remains to show that $\dim_2(W_{1,n})\le\left\lceil n/2\right\rceil$ and $\dim_3(W_{1,n})\le n-\left\lfloor n/5\right\rfloor$.
With this aim,  let $A_k\subset V(C_n)$,  $k\in \{2,3\}$, be a set of vertices such that $u_i$ belongs to $A_2$ or $A_3$  if and only if $i$ is odd or $i\not\equiv 0\;(5)$.  Notice that $|A_2|=\left\lceil n/2\right\rceil$ and  $|A_3|=n-\left\lfloor n/5\right\rfloor$. We will show that for every $u_i,u_j\in V(C_n)$, $i\ne j$, it hold $|\mathcal{D}_{W_{1,n}}(u_i,u_j)\cap A_k|\ge k$ and then, by Lemmas \ref{lemmaW1n-K_Plus_2} and \ref{Any-KGenerator-K1H_k}, we will have that $A_k$ is a $k$-metric generator for $W_{1,n}$. Consider each $A_k$ separately:
\\
\\
\noindent $k=2$. If $u_i,u_j\in A_2$, then the result is straightforward. If $u_i\in A_2$ and $u_j\not\in A_2$, then $\{u_i,u_k\}\subseteq A_2\cap\mathcal{D}_{W_{1,n}}(u_i,u_j)$, for some $u_k\in N(u_j)-N[u_i]$. Also, if $u_i,u_j\not\in A_2$, then $\{u_k,u_l\}\subseteq A_2\cap\mathcal{D}_{W_{1,n}}(u_i,u_j)$, where $u_k,u_l\in N(u_i)\triangledown N(u_j)$.
\\
\\
\noindent  $k=3$. If $u_i,u_j\in A_3$, then $\{u_i,u_j,u_k\}\subseteq A_3\cap\mathcal{D}_{W_{1,n}}(u_i,u_j)$, where $u_k\in A_3\cap(N[u_i]\triangledown N[u_j])$. If $u_i\in A_3$ and $u_j\not\in A_3$, then $\{u_i,u_k,u_l\}\subseteq A_3\cap\mathcal{D}_{W_{1,n}}(u_i,u_j)$, where $u_k,u_l\in A_3\cap(N[u_j]\triangledown N[u_i])$. Finally, if $u_i,u_j\not\in A_3$, then $\{u_k,u_l,u_m\}\subseteq A_3\cap\mathcal{D}_{W_{1,n}}(u_i,u_j)$, where $u_k,u_l,u_m\in N(u_i)\cup N(u_j)$.

Therefore, $A_k$ is a $k$-metric generator for $W_{1,n}$, with $k\in \{2,3\}$ and, as a consequence, the result follows.
\end{proof}

Finally, we present the relationship between $\dim_k(G\odot\mathcal{H})$ and $\dim_k(G\odot(K_1\diamond\mathcal{H}))$ for a family $\mathcal{H}$ of cycles of order greater than six and $k\in \{1,2,3,4\}$. We only consider $k\in \{1,2,3,4\}$, since for $n'\ge 7$ we have that $\mathcal{C}(C_{n'})=\mathcal{C}(W_{1,n'})=4$, and as a consequence, by Corollary \ref{coroKmetricCorona}, $G\odot\mathcal{H}$ and $G\odot(K_1\diamond\mathcal{H})$ are $4$-metric dimensional. Thus, by Theorem \ref{theoEqualCoronaDimH} and Propositions \ref{propEqH-K1_H} and \ref{value-wheels}, we obtain the following result.

\begin{proposition}
Let $G$ be a connected graph of order $n\ge 2$ and let $\mathcal{H}$ be a family of $n$ cycles. If every cycle $C_i\in\mathcal{H}$ has order $n_i\ge 7$, then
\begin{enumerate}[{\rm (i)}]
\item  $\dim(G\odot\mathcal{H})=\dim(G\odot\mathcal{\overline{H}})=\dim(G\odot(K_1\diamond\mathcal{H}))=\sum_{i=1}^n\left\lfloor(2n_i +2)/5\right\rfloor $.
\item $\dim_2(G\odot\mathcal{H})=\dim_2(G\odot\mathcal{\overline{H}})=\dim_2(G\odot(K_1\diamond\mathcal{H}))=\sum_{i=1}^n\left\lceil n_i/2\right\rceil$.
\item $\dim_3(G\odot\mathcal{H})=\dim_3(G\odot\mathcal{\overline{H}})=\dim_3(G\odot(K_1\diamond\mathcal{H}))=\sum_{i=1}^n\left(n_i-\left\lfloor n_i/5\right\rfloor\right)$.
\item $\dim_4(G\odot\mathcal{H})=\dim_4(G\odot\mathcal{\overline{H}})=\dim_4(G\odot(K_1\diamond\mathcal{H}))=\sum_{i=1}^nn_i$.
\end{enumerate}
\end{proposition}

\end{document}